\documentclass[11pt]{amsart}
\usepackage{graphicx}
\usepackage{amsmath}
\usepackage{bbm}
\usepackage{amssymb}
\usepackage{color}
\usepackage{mathdots}
\usepackage{amsfonts}
\usepackage{mathrsfs}
\usepackage{mathtools}

\newtheorem{thm}{Theorem}[section]
\newtheorem{cor}[thm]{Corollary}

\newtheorem{lem}[thm]{Lemma}
\newtheorem{prop}[thm]{Proposition}
\theoremstyle{definition}
\newtheorem{defn}[thm]{Definition}
\theoremstyle{remark}
\newtheorem{rem}[thm]{Remark}

\numberwithin{equation}{section}

\newcommand{\cI}{\mathcal{I}}

\newcommand{\RR}{\mathbb{R}}

\newcommand{\CC}{\mathbb{C}}
\newcommand{\NN}{\mathbb{N}}

\newcommand{\QQ}{\mathbb{Q}}

\newcommand{\powerset}{\raisebox{.15\baselineskip}{\Large\ensuremath{\wp}}}
\makeatletter
\def\@tocline#1#2#3#4#5#6#7{\relax
  \ifnum #1>\c@tocdepth 
  \else
    \par \addpenalty\@secpenalty\addvspace{#2}%
    \begingroup \hyphenpenalty\@M
    \@ifempty{#4}{%
      \@tempdima\csname r@tocindent\number#1\endcsname\relax
    }{%
      \@tempdima#4\relax
    }%
    \parindent\z@ \leftskip#3\relax \advance\leftskip\@tempdima\relax
    \rightskip\@pnumwidth plus4em \parfillskip-\@pnumwidth
    #5\leavevmode\hskip-\@tempdima
      \ifcase #1
       \or\or \hskip 1em \or \hskip 2em \else \hskip 3em \fi%
      #6\nobreak\relax
    \dotfill\hbox to\@pnumwidth{\@tocpagenum{#7}}\par
    \nobreak
    \endgroup
  \fi}
\makeatother

\begin{document}
\title[Riesz variational principle]{Moment indeterminateness: the Marcel Riesz variational principle
} %

\author[D. P. Kimsey]{David P. Kimsey}
\address{(DPK) School of Mathematics and Statistics\\
Newcastle University\\
Newcastle upon Tyne NE1 7RU UK}
\email{david.kimsey@newcastle.ac.uk}

\author[M. Putinar]{Mihai Putinar}
\address{(MP) Department of Mathematics \\
University of California at Santa Barbara \\
Santa Barbara, CA 93106-3080 USA and \\
School of Mathematics and Statistics\\
Newcastle University\\
Newcastle upon Tyne NE1 7RU UK}
\email{mputinar@math.ucsb.edu}

\begin{abstract} 
The discrete data encoded in the power moments of a positive measure, fast decaying at infinity on euclidean space, is incomplete for recovery, leading to the concept of moment indeterminateness. On the other hand, classical integral transforms (Fourier-Laplace, Fantappi\`e, Poisson) of such measures are complete, often invertible via an effective inverse operation. The gap between the two non-uniqueness/ uniqueness phenomena is manifest in the dual picture, when trying to extend the measure, regarded as a positive linear functional, from the polynomial algebra to the full space of continuous functions. This point of view was advocated by Marcel Riesz a century ago, in the single real variable setting. Notable advances in functional analysis have root in Riesz' celebrated four notes devoted to the moment problem. A key technical ingredient being there the monotone approximation by polynomials of kernels of integral transforms. With inherent new obstacles we reappraise in the context of several real variables M. Riesz' variational principle. The result is an array of necessary and sufficient moment indeterminateness criteria, some raising real algebra questions, others involving intriguing analytic problems,  all gravitating around the concept of moment separating function.
\end{abstract}

\maketitle

\section{Introduction}

Traditionally, the inverse problem of estimating a mass distribution from power moment data on the real line emerged during the second half of XIX-th Century from best uniform approximation questions and continued fraction expansions. The advance of integration theory and functional analysis, both partially motivated by the very moment problem we refer to, revealed new facets of it, and more importantly forged novel mathematical tools. In this respect,
the four notes Marcel Riesz wrote between 1922 and 1923, devoted to the moment problem on the line, burst of original ideas and far reaching new results
\cite{Riesz-collected}. One of the widely circulated innovations contained in these notes is M. Riesz' extension lemma of positive linear functionals, leading to the construction of the integral of discontinuous (measurable) functions with respect to a positive measure. This non-constructive extension of linear functionals is similar, arguably more intuitive, and precedes the not less celebrated Hahn-Banach lemma \cite{collective}.

The aim of the present note is to isolate and exploit from Riesz' ``Troisieme Note" \cite{Riesz-3} a few pertinent observations which shed light on the intricate nature of 
moment indeterminateness in several variables. The vantage of today's integration theory, in particular the construction of Daniell's integral, simplifies the language, offering a novel perspective on the genuine complications of multivariate indeterminateness.

We focus below on a euclidean space setting and polynomial functions, although a generalization to abstract measures defined on a locally compact space and a prescribed space of initial test functions is not more demanding.
Let $d \geq 1$ be a fixed dimension. We denote  $x^\alpha = x_1^{\alpha_1} x_2^{\alpha_2} \ldots x_d^{\alpha_d}$ for a point $x = (x_1, x_2, \ldots,x_d) \in \RR^d$ and multi-index $\alpha = (\alpha_1, \alpha_2, \ldots, \alpha_d) \in \NN_0^d.$ The euclidean norm is $\| x \|^2 = x_1^2 + x_2^2 + \ldots + x_d^2.$
The algebra of polynomials with real coefficients will be denoted by $\RR[x_1, \ldots, x_d]$ or simply $\RR[x]$. If necessary, we will pass to complex coefficients, writing then
$\CC[x_1, \ldots, x_d]$.

Let $X \subset \RR^d$ be a closed subset and denote by $C_p(X)$ the space of real valued, continuous functions on $X$, having polynomial growth at infinity. That is, a continuous function $f : X \longrightarrow \RR$ belongs to $C_p(X)$ if there exists a positive integer $n$ and a positive constant $C$, such that
$$ |f(x)| \leq C (1 + \| x \|)^n, \ \ x \in X.$$
The vector space $C_p(X)$ is an inductive limit of normed spaces, with semi-norms given by the best constants $C$ appearing above. In view of F. Riesz' Theorem, and its offsprings, a positive linear functional $L : C_p(X) \longrightarrow \RR$ is represented by a finite, positive Radon measure $\mu$ supported by $X$:
$$ L(f) = \int_X f d\mu, \ \ f \in C_p(X).$$
See for instance \cite{Bourbaki} Chapters 3 and 4, where integration on locally compact spaces is treated in full detail. The natural restriction map
$\RR[x_1, \ldots, x_d] \longrightarrow C_p(X)$ defines a subalgebra, denoted $\RR[X]$, and called the set of polynomial functions on $X$. Notice that this restriction map may not be injective.
The Radon measure $\mu$ associated to a linear, positive functional $L \in C_p(X)^\ast$ is called {\it admissible}, reflecting the classical assumption that
all polynomial functions, and hence power moments, are integrable:
$$ \int_X |p| d\mu  <\infty, \ \ p \in \RR[x_1, \ldots, x_d].$$ Two such measures $\mu_1, \mu_2$ are called {\it moment equivalent}, or throughout this article simply {\it equivalent} if
$$  \int_X p d\mu _1 =  \int_X p d\mu_2, \ \ p \in \RR[x_1, \ldots, x_d].$$
In this case, we shall write $\mu_1 \sim_X \mu_2$, or more simply $\mu_1 \sim \mu_2$ when no confusion can possibly arise. If there exist two distinct measures $\mu_1$ and $\mu_2$ such that $\mu_1 \sim_X \mu_2$, then we will say that $\mu_1$ is {\it $X$-indeterminate}. If $X = \RR^d$, then we will simply say that $\mu_1$ is {\it indeterminate}.

We translate below, in the modern setting, the core of section 6 (``La question d'unicit\'e. Crit\`eres pr\'eliminaires") of \cite{Riesz-3}.

\begin{thm}\label{basic}[M. Riesz, 1923] Let $X$ denote a closed subset of $\RR^d$ and let $L : \RR[x] \longrightarrow \RR$ be a linear functional. Then

1) The functional $L$ is representable by an admissible measure $\mu$ supported on $X$ if and only if:
\begin{equation}\label{positive}
(f \in \RR[x], \ \ f \geq 0) \ \implies (L(f) \geq 0).
\end{equation}

2) There exists another admissible measure $\nu$, moment equivalent and distinct to the representing measure $\mu$ if and only if
there exists a function $\phi \in C_p(X) \setminus \RR[x]$ satisfying
\begin{equation}\label{separating}
\sup_{p\leq \phi} L(p) < \inf_{q \geq \phi} L(q).
\end{equation}
\end{thm}

The proof relies on the cited extension lemma. Namely, start with a continuous function of polynomial growth on $X$, $\phi \in C_p(X)$.
There are non-trivial polynomial functions $p, q \in \RR[x]$, such that $p \leq \phi \leq q$ on $X$. In particular, if the functional $L$ is positive, then  $L(p) \leq L(q)$. Choose {\it any} value $t \in [\sup_{p\leq \phi} L(p),  \inf_{q \geq \phi} L(q)].$ The linear extension of the functional
$$ \tilde{L}( p + \lambda \phi) = L(p) + \lambda t, \ \ \lambda \in \RR,$$
turns out to be positive on the vector space $\RR[x] + \RR \phi$. A maximal element in a chain of such extensions, ordered by inclusion, exists by Zorn's Lemma, and it is necessarily equal to the full space $C_p[X]$. If the interval $[\sup_{p\leq \phi} L(p),  \inf_{q \geq \phi} L(q)]$ does not reduce to a point, then
one can choose different extension of the functional $L$, and hence different representing measures. Conversely, if there are two distinct measures $\mu$ and $\nu$ possessing the same values on polynomial functions, then there exists a continuous function $\phi \in C_p[X]$ satisfying
$$ \int_X \phi d\nu \neq \int_X \phi d\mu.$$
Then these two distinct values belong to the interval $[\sup_{p\leq \phi} L(p),  \inf_{q \geq \phi} L(q)].$ 

Marcel Riesz' extension technique was exploited (and sometimes reinvented over decades) in relation to the construction of various integrals and measures. Notable in this respect is Daniell's integral \cite{Cotlar,Bourbaki}. A general account on extensions of positive functionals and relations to specific integrals is contained in Metzler's articles \cite{Metzler-1, Metzler-2}. For applications to moment problems we refer to
Akhiezer's book \cite{Akhiezer-book} Chapter 6, Section 6, and also \cite{AK}, pg. 137, or \cite{Arens,Butzer}.

\begin{defn} A function $\phi$ appearing in condition (\ref{separating}) is called {\it separating} with respect to the positive functional $L$.
\end{defn}

\begin{rem}
The subtitle of section 6 in \cite{Riesz-3} very accurately labels the above theorem as only containing ``preliminary criteria" of existence and uniqueness. A central object, well studied in connection with moment problems on the real line ($d=1$) by at least two generations of mathematicians before Riesz is the so called Christoffel function; in our notation
$$ \rho(\alpha) = \inf_{p(\alpha)=1} L(p^2), \ \ \alpha \in \CC.$$
The determinateness criterion $\rho(\alpha) = 0$ for at least one (and then, all) non-real value $\alpha$ appeared implicitly in the works of Stieltjes and Hamburger, see \cite{Akhiezer-book} for details. By a true tour de force, exposed in section 24 (``\'Etude aprofondie de fonction $\rho(\alpha)$") of \cite{Riesz-3}, M. Riesz proves, in the indeterminate case on the line, that the function $\frac{1}{\rho(\alpha)}$ has a sub-exponential growth, and the integral
$$ \int_{-\infty}^\infty \frac{ \log \frac{1}{\rho(t)}}{1+t^2} dt$$
is absolutely convergent. His deep results preview a full solution to S. Bernstein's weighted, uniform approximation of continuous functions on the entire line,
obtained only a few decades later, cf. \cite{Akhiezer}. Returning to the indeterminate case of the moment problem on $\RR$, a corollary of Riesz' calculations is that {\it any} non-polynomial function is separating for the respective functional. Based on this observation, effective determinateness criteria for the moment problem on the line are deduced, including the well known Carleman criterion, cf. \cite{Riesz-3} pg. 48. 

Regardless to say that, allowing the extension of the positive functional $L : \RR[x] \longrightarrow \RR$ to more general functions, such as semi-continuous ones will enlarge the pool of separating elements in condition  (\ref{separating}), and possibly simplify it. We will return to this observation in the next section.
\end{rem}

The main theme of this article is to extract from various faithful transforms of measures some elementary separating functions for indeterminate, multivariate moment problems. Positivity certificates making our necessary and sufficient conditions for moment indeterminateness more effective are not available at this stage.

A traditional approach, complementary to Riesz' variational principle, relates the multivariate moment problem to the spectral decomposition of strongly commuting tuples of symmetric,
generally unbounded, Hilbert space operators. While the construction in this setting of the joint spectral measure goes hand in hand with the extension of positive linear functionals, the
monotonic approximation of separating functions we propose opens a new vista towards real algebra. For the Hilbert space interpretation of the moment problem we refer to 
Akhiezer's monograph \cite{Akhiezer-book} and the recent book by Schm\"udgen \cite{Sch_book}. Equally relevant, and neglected in our article, is the link between moment indeterminateness and the still mysterious topics of quasi-analytic classes in several variables, see for instance \cite{Ronkin}.

The contents of the article are as follows. Section 2 brings  trigonometric separating functions to the forefront via the Fourier-Laplace transforms. Section 3 deals with discontinuous separating functions. Section 4 relies on Poisson's transform to render a quantitative criteria for indeterminateness. Various operations preserving indeterminateness are analyzed in Section 5. In Section 6, we relate indeterminateness criteria to the existence of bounded point evaluations, a traditional theme in the one variable setting. In Section 7  we focus on moment problems supported by real algebraic, affine curves, with special emphasis on rational curves.

The present article does not refer to the spectral analysis interpretation of the multivariate moment problem. This subject is amply exposed in the monograph \cite{Sch_book}.

\section{Fourier-Laplace transform}

We start with a full space scenario. Let $\mu$ be an admissible measure on $\RR^d$.
For $x = (x_1, \ldots, x_d) \in \RR^d$ and $\xi = (\xi_1, \ldots, \xi_d) \in \RR^d$, we denote $x \cdot \xi = \prod_{j=1}^d x_j \xi_j$ and $\| x \| = \sqrt{x \cdot x}$. The {\it Fourier transform of $\mu$},
$$\hat{\mu}(\xi) := \int_{\RR^d} e^{-i x \cdot \xi} d\mu(x) \quad \quad {\rm for} \quad \xi \in \RR^d$$
is a smooth function. Indeed, since all power moments of $\mu$ are finite,
$$\left|  \left( \frac{ \partial^{\lambda} }{d\xi^{\lambda} } \, \hat{\mu} \right)(\xi) \right| = \left| \int_{\RR^d} x^{\lambda} e^{-i x\cdot \xi} d\mu(x) \right|
\leq \int_{\RR^d} |x^{\lambda} | d\mu(x) < \infty.$$
In addition, according to Bochner's theorem the function $\hat{\mu}$ is positive definite, i.e., the continuous kernel function is positive semi-definite:
$$(\hat{\mu}(\xi_i - \xi_j) )_{i,j=1}^n \succeq 0$$
for any choice of $\xi_1, \ldots, \xi_n \in \RR^d$.

\begin{defn}
For $\xi\in  \RR^d \setminus \{ 0 \}$, we define the {\it push-forward measure} $\mu_{\xi}$ as
$$\int_{\RR} \varphi(t) d\mu_{\xi}(t) = \int_{\RR^d} \varphi(x \cdot \xi) d\mu(x)$$
for every continuous function $\varphi$ of polynomial growth at infinity. 
\end{defn}

Assume $\mu$ and $\nu$ are two distinct, moment equivalent measures, in which case the Fourier transforms of $\mu$ and $\nu$ are distinct, see for instance \cite{Palamodov}. Therefore
there exists $\xi \neq 0$ with the property
$$ \int e^{-i x \cdot \xi} d\mu(x) \neq  \int e^{-i x \cdot \xi} d\nu(x).$$
Moreover, the continuity of the Fourier transform implies that there exists $\delta>0$ so that
 $$ \int e^{-i x \cdot \eta} d\mu(x) \neq  \int e^{-i x \cdot \eta} d\nu(x),$$
for all $\eta, |\eta - \xi| < \delta$.  Riesz' Theorem \ref{basic} implies the following theorem.

\begin{thm}\label{Fourier} Let $\mu$ be a positive measure on $\RR^d$ with finite moments of any order. There exists a different positive measure 
on $\RR^d$ with the same power moments if and only if there exists $\xi \in \RR^d \setminus \{0\}$ and $\epsilon \in \{0, -\pi/2\}$ with the property
$$ \sup_{p(x) \leq \cos( x \cdot \xi + \epsilon)} \int p d\mu <  \inf_{q(x) \geq \cos( x \cdot \xi + \epsilon)} \int q d\mu,$$
where $p, q \in \RR[x_1, \ldots, x_d]$ are polynomials.

Moreover, the above separation condition is open in $(\xi,\epsilon)$.
\end{thm}

\begin{rem}
\label{rem:pushforward}
By pushing forward the competing polynomials in the above variational inequality on the line spanned by the vector $\xi$, with $\| \xi \| = 1$, we find a much weaker
condition
$$ \sup_{r(x \cdot \xi) \leq \cos( x \cdot \xi + \epsilon)} \int r d\mu <  \inf_{s(x \cdot \xi) \geq \cos( x \cdot \xi + \epsilon)} \int s d\mu,$$
with $r, s \in \RR[t]$ univariate polynomials. In both cases, we encounter the largely open task of certifying inequalities involving polynomials and trigonometric functions.
\end{rem}

Potentially a more accessible context is offered by a bilateral, or traditional, Laplace transform, with output complex variables. To this aim
we recall the {\it Fantappi\`e transform} of an admissible measure $\mu$:

\begin{equation}
\label{eq:FT}
F_{\mu}(z, \xi) = \int_{\RR^d} \frac{ d\mu(x) }{ x \cdot \xi - z }, \ \ z \in \CC, \ {\rm Im}\, z > 0, \xi \in \RR^d.
\end{equation}

Note that the above integral is convergent since $|x \cdot \xi - z| \geq {\rm Im}\, z > 0$.
 Fantappi\`e's transform is an iterated Fourier-Laplace transform, hence invertible. Indeed, 
$$
\int_0^{\infty} e^{-ip [ x \cdot \xi - z] }dp = \frac{1}{i(x \cdot \xi - z)},
$$
and Fubini's theorem yield
$$F_{\mu}(z, \xi) = i \int_0^{\infty} \int_{-\infty}^{\infty} e^{-ip x \cdot \xi} e^{i p z} d\mu(x) dp = i \int_0^{\infty} e^{i p z} \hat{\mu}(p \xi) dp.$$

The Fantappi\`e transform $F_{\mu}(z, \xi)$ is homogeneous of degree $-1$, i.e., 
$$F_{\mu}(t z, t \xi) = t^{-1} F_{\mu}(z, \xi) \quad \quad {\rm for} \quad t \in \RR \setminus \{ 0 \}.$$
Therefore, the values $F_{\mu}(z,w)$, where $\| w \| = 1$ and ${\rm Im}\, z > 0$ determine $F_{\mu}$, and $\mu$.
In complete analogy to Theorem \ref{Fourier} we state the following indeterminateness criterion.

\begin{thm} Let $\mu$ be an admissible measure on $\RR^d$. There exists a moment equivalent, admissible measure, distinct of $\mu$ if and only if
there exists $w \in \RR^d, \ \|w\|=1,$ and $\epsilon \in \{0,1\}$, such that
\begin{equation}
 \sup_{p(x) \leq  \frac{ (1+ \epsilon w \cdot x)}{ (w \cdot x)^2 + 1}} \int p d\mu <  \inf_{q(x) \geq \frac{ (1+ \epsilon w \cdot x)}{ (w \cdot x)^2 + 1}} \int q d\mu,
 \end{equation}
where $p,q$ are polynomials.
\end{thm}

\begin{proof}

Assume $\mu \sim \nu$. Then $\mu \neq \nu$ if and only if there exists $w \in S^{d-1}$ such that
$$
\int_{\RR^d} \frac{d\mu_1(x) }{x \cdot w - z} \neq \int_{\RR^d} \frac{d\mu_2(x) }{x \cdot w - z}
$$
where $w \in S^{d-1}$ and $z \in \CC$ with ${\rm Im}\, z > 0$. In other terms, the push-forward measures have distinct Cauchy transforms at the specific point $z$:
$$\int_{\RR} \frac{d\mu_{w}(t)}{t- z} \neq \int_{\RR} \frac{d\nu_{w}(t) }{t -z}.$$
But the measures $\mu_w, \nu_w$ are moment equivalent on the line. By well known results of one variable theory (specifically the parametrization of Weyl's circle by values of Cauchy transforms of representing measures), one also finds
$$\int_{\RR} \frac{d\mu_{w}(t)}{t- i} \neq \int_{\RR} \frac{d\nu_{w}(t) }{t -i},$$
see Section 2.1.2 in \cite{Akhiezer-book}.

By taking real, respectively imaginary, parts of $\frac{1}{ w \cdot x -i}$ one recovers the separating functions in the statement.

\end{proof}

Fantappi\`e's transform is particularly useful for measures supported on a convex cone. In this context, complex analyticity and complete monotonicity properties enhance the characterization of the range of the transform and provide efficient inversion formulae. We refer to \cite{HenkinShananin} for full details.
Next we extract a few relevant observation from the theory of Fantappi\`e transform on convex cones. We can regard the analysis below as an analogue of Stieltjes moment problem on the real semi-axis.

Let
$\Gamma \subseteq \RR^d$ be an acute, convex and solid cone and $\Gamma^* = \{ \eta \in \RR^d: \eta \cdot x \geq 0 \quad {\rm for} \;\; {\rm all} \quad x \in \Gamma \}$ be the dual cone of $\Gamma$. Let $\mu$ be an admissible measure supported on $\Gamma^*$.  Note that in this case the Fantappi\'e transform 
$$
F_{\mu}(z, w) = \int_{\Gamma^*} \frac{d\mu(x)}{w \cdot x - z}
$$
admits a complex analytic extension to the domain
$${\rm Re}\, w \in \Gamma \quad {\rm and} \quad {\rm Re}\, z < 0.$$

In particular the range of real values $w \in \Gamma, \  z <0,$ is a uniqueness set for the complex analytic function $F_\mu$ defined on the tube domain over this convex set. In short, due to homogeneity, the values
$$ 
F_{\mu}(-1, a) = \int_{\Gamma^*} \frac{d\mu(x)}{a \cdot x +1}, \ \ a \in \Gamma,
$$ determine the measure $\mu$. Moreover, since the above function is real analytic in the variable $a$, a non-trivial zero set of a
difference $F_{\mu}(-1, a) -F_{\nu}(-1, a)$ is a proper analytic subset of ${\rm int} \Gamma$.

Mutatis mutandis, the following result is proved.

\begin{thm}\label{CONE}
Let $\Gamma \subseteq \RR^d$ be an acute, convex and solid cone and let $\mu$ be an admissible measure supported by the dual cone $\Gamma^\ast$.
There exists a different, admissible measure supported on $\Gamma^\ast$ and moment equivalent to $\mu$ if and only if there exists $a \in {\rm int} \Gamma$, such that
\begin{equation}
 \sup_{\stackrel{p(x) \leq  \frac{ 1}{ a \cdot x + 1}} {x \in \Gamma^\ast}} \int p d\mu <  \inf_{\stackrel{q(x) \geq  \frac{ 1}{ a \cdot x + 1}} {x \in \Gamma^\ast}} \int q d\mu,
 \end{equation}
 where $p,q$ are polynomials functions on $\Gamma^\ast$.
 
 Moreover, the range of values of $a$ above is an open, everywhere dense subset of  ${\rm int} \, \Gamma$.
\end{thm}

\section{Discontinuous separating functions}

Let $X \subset \RR^d$ be a closed set. Working with test functions of polynomials growth at infinity imposes the following adaptation of the class of Baire-1 functions. We refer to \cite{Valee-Poussin} for the classical setting.

\begin{defn}
A function $f : X \longrightarrow \RR$ is called a {\it Baire function of the first category and of polynomial growth}, in short $f \in {\mathcal{BP}}_1(X)$, if there exists a sequence
$\phi_n \in C_p(X)$ subject to a uniform bound $(M>0, N \geq 0)$:
 \begin{equation}\label{pointwise}
 |\phi_n(x)| \leq M(1+ \| x \|^2)^N, \ \ n \geq 1,\ x \in X,
 \end{equation}
such that, pointwisely, 
$$ f(x) = \lim_{n \rightarrow \infty} \phi_n(x), \ \ x \in X.$$
\end{defn}

In particular a semi-continuous function $f$ defined on $X$ and of polynomial growth satisfies $f \in {\mathcal{BP}}_1(X)$. Indeed, assume
$f$ is lower semi-continuous and 
$$ |f(x)| \leq \rho(x) := M(1+ \| x \|^2)^N, x \in X,$$
for some constants $M>0, N\geq 0$. There exists a monotonically increasing sequence of continuous functions $\phi_n$ converging pointwisely to $f$.
Then the sequence $\phi'_n = \min(\phi_n, \rho)$ is uniformly bounded from above by the weight $\rho$ and converges poinwisely to $f$. Similarly,
the operation $\phi_n = \max(\phi'_n, - 2\rho)$ provides a lower bound of polynomial decay at infinity, without altering the convergence behavior.

\begin{thm} Let $X \subset \RR^d$ be a closed subset and let $\mu$ be an admissible measure on $X$.
$\mu$ is $X$-indeterminate if and only if $\mu$ has a non-polynomial separating function $f \in {\mathcal{BP}}_1(X)$.
\end{thm} 

\begin{proof} The essence of the construction of Daniell's integral is: a linear, positive functional on $C_p(X)$ admits a {\it unique} extension to a linear, positive functional on ${\mathcal{BP}}_1(X)$. Indeed, given the admissible Radon measure $\mu$, if a sequence of continuous functions $\phi_n$ converges to $f$ as stated in (\ref{pointwise}), then Lebesgue dominated convergence theorem
implies $\int f d\mu = \lim_n \int \phi_n d\mu$. A detailed analysis of this uniqueness of extension phenomenon is contained in  \cite{Cotlar,Metzler-2}.

In other terms, two admissible measures $\mu, \nu$, moment equivalent on $X$, are separated by the function $f$ above:
$$ \int f d\mu \neq \int f d\nu$$ if and only if $$\int \phi_n d\mu \neq \int \phi_n d\nu,$$ for $n$ large. Translating this observation into M. Riesz' framework, we infer: $\mu$ is $X$-indeterminate if and only if there exists a non-polynomial function $f \in {\mathcal{BP}}_1(X)$ subject to:
$$ \sup_{p \leq f} \int p d\mu < \inf_{q \geq f} \int q d\mu,$$
where $p, q$ are polynomials, and the inequalities are restricted to points of $X$.
\end{proof}

Note that the uniqueness of extension of a positive linear functional does not extend beyond Baire first class. Indeed, Dirichlet's characteristic function
$\chi = \chi_{\QQ}$ of rational numbers in the interval $[0,1]$ yields, for Riemann's integral:
$$ \sup_{p \leq \chi} \int_0^1 p(x) dx \leq 0 < 1 \leq \inf_{q \geq \chi} \int_0^1 q(x)  dx$$ 
and the measure $dx$ is determined on $[0,1]$.

A simple application of this extension of the field of test, separating function is given by the multivariate distribution function associated to a measure.
To this aim, for a point $a \in \RR^d$, we define the semi-bounded orthant
$$ \Pi_a = \{ x \in \RR^d: \ x_j \leq a_j \quad {\rm for} \quad j=1, \ldots ,d\}.$$
We denote by $\chi_a = \chi_{\Pi_a}$ the characteristic function of this set.

\begin{cor} Let $X \subset \RR^d$ be a closed subset and let $\mu$ be an admissible measure on $X$. 
The measure $\mu$ is $X$-indeterminate, if and only if there exists $a \in \RR^d$ such that
$$ \sup_{\stackrel{p(x) \leq \chi_a(x)}{ x \in X}} \int p d\mu < \inf_{\stackrel{q(x) \geq \chi_a(x)}{x \in X}} \int q d\mu.$$
\end{cor}

\begin{proof} The linear subspace of ${\mathcal{BP}}_1(X)$ contains the characteristic functions of all rectangles
of the form $(b_1, a_1] \times (b_2, a_2] \times \ldots \times (b_d, a_d]$. Moreover, it is well known that the associated step functions
are a set of uniqueness for any Radon measure on $X$, see, e.g., \cite{Cotlar}.
\end{proof}

%

Modifications of separating functions become more flexible in the class of Baire-1 functions. We state only a simple observation in this direction.

\begin{prop} Let $X \subset \RR^d$ be a closed subset and let $\mu$ be an admissible, $X$-indeterminate measure. Denote by $S \subset X$ the closed support of $\mu$.
Let $\phi \in C(X)$ be a continuous, separating function from the left, with respect to $\mu$, specifically:
$$ \inf_{\stackrel{p(x) \leq \phi(x)}{ x \in X}}  \int (\phi-p) d\mu > 0.$$
Let $\psi \in C(X)$ be a continuous function satisfying $\psi(x) \leq \phi(x), \ x \in X \setminus S$.
Then the Baire-1 function 
$$ \Phi(x) = \left\{ \begin{array}{l}
                          \phi(x), \ x \in S,\\
                          \psi(x), \ x \in X \setminus S,
                          \end{array} \right. $$
is left separating for $\mu$.
                        
                          \end{prop} 
\begin{proof} Note that the function $\Phi$ is upper-semicontinuous, therefore $\Phi$ is Baire-1.
Let $p$ be a polynomial function satisfying  $p(x) \leq \Phi(x), \ x \in X$. Then $p(x) \leq \phi(x), \ x \in X$, hence
\begin{align*}
  \inf_{\stackrel{p(x) \leq \Phi(x)}{ x \in X}}  \int_X (\Phi-p) d\mu =& \;  \inf_{\stackrel{p(x) \leq \Phi(x)}{ x \in X}}  \int_S (\phi-p)  d\mu \\
  \geq& \;   \inf_{\stackrel{p(x) \leq \phi(x)}{ x \in X}}  \int (\phi-p) d\mu  > 0.
  \end{align*}
\end{proof}  

The restriction of separation from the left in the  statement is not essential, since either
$$    \inf_{\stackrel{p(x) \leq \phi(x)}{ x \in X}}  \int (\phi-p) d\mu > 0,$$ 
or
 $$ \inf_{\stackrel{q(x) \geq \phi(x)}{ x \in X}}  \int (q-\phi) d\mu > 0.$$ 
Moreover, adding a polynomial to a function of polynomial growth will not change its $\mu$ separating status. In particular, assuming $\phi \geq 0$ in the 
 statement above, one can take  $\psi = 0$.

We elaborate below the important case of characteristic functions of coordinate quadrants, seen as separating functions.
Suppose $h \in \RR[t]$ satisfies
\begin{equation}\label{semiaxis}
\begin{cases}
h(t) \geq 1 & {\rm for} \quad t \geq 0, \\
h(t) \geq 0 & {\rm for} \quad t \in \RR.
\end{cases}
\end{equation}
Such polynomials are usually obtained via quadrature formulas, see for instance \cite{Freud} Section II.3.

Let $Q = \{ (x_1, \ldots, x_d) \in \RR^d: x_j \geq 0 \quad { \rm for} \quad j=1,\ldots, d \}$ and consider
$$H(x_1, \ldots, x_d) = \prod_{j=1}^d h(x_j)$$
and notice that $H \geq \chi_Q$ on $\RR^d$, i.e.,
$$
\begin{cases}
H(x) \geq 1 & {\rm for} \quad x  \in Q, \\
H(x) \geq 0 & {\rm for} \quad x \in \RR^d \setminus Q.
\end{cases}
$$
 For a subset $\cI \subseteq \{ 1, \ldots, d \}$, we let
 $$Q_{\cI} = \{ (x_1, \ldots, x_d) \in \RR^d: \text{$x_j \geq 0$ for $j \notin \cI$ and $x_k \leq 0$ for $k \in \cI$}\}.$$
 In a similar fashion for a function $\varphi$, we let
 $$\varphi_{\cI}(x_1, \ldots, x_d)= \varphi(x_1, \ldots, x_{k-1}, -x_k, x_{k+1}, \ldots, x_d).$$
Note that, the polynomial $H_{\cI}$ satisfies
$$
\begin{cases}
H_{\cI}(x) \geq 1 & {\rm for} \quad x  \in Q_{\cI}, \\
H_{\cI}(x) \geq 0 & {\rm for} \quad x \in \RR^d \setminus Q_{\cI}.
\end{cases}
$$
and $H_1 := H_{ \{ 1 \} } =  \sum_{\cI \neq \emptyset} H_{\cI}$, where the sum is taken over all nonempty subsets of $\{ 1, \ldots, d \}$, satisfies
$$
\begin{cases}
H_{1}(x) \geq 1 & {\rm for} \quad x  \in \RR^d \setminus Q, \\
H_{1}(x) \geq 0 & {\rm for} \quad x \in Q.
\end{cases}
$$

Therefore, think of $\chi_Q$ as a separating function, i.e.,
$$1 - H_1 \leq \chi_Q \leq H.$$
Let $\mu$ be an admissible measure on $\RR^d$ with $\chi_Q$ as a separating function, i.e.,
$$\inf_{g \leq \chi_Q \leq f } \int (f -g) \, d\mu \geq \gamma >0,$$
where $f,g \in \RR[x_1, \ldots, x_d]$. We infer 
$$\int (H-(1-H_1) ) \, d\mu \geq \gamma,$$
i.e.,
$$\sum_{\cI \in \powerset\{ 1, \ldots, d\} } \int H_{\cI} d\mu - \mu(\mathbbm{1}) \geq \gamma > 0 ,$$
where $\mu(\mathbbm{1}) := \int \, d\mu$.

We can, of course, translate $Q$ via $Q \mapsto a+Q$, for $a \in \RR^d$, where 
$$a + Q := \{ y \in \RR^d: \text{$y = a+ (x_1, \ldots x_d)$ for some $(x_1, \ldots, x_d \in Q$}\},$$
which amounts to considering the function
$$(\tau_a \, \varphi)(x) = \varphi(x-a).$$

All in all, we can state the following indeterminateness necessary condition.

\begin{prop} Let $h(t)$ be a polynomial satisfying \ref{semiaxis} and let $H(x_1, \ldots, x_d) = \prod_{j=1}^d h(x_j)$.
Assume $\mu$ is an indeterminate, admissible measure defined on $\RR^d$. Then there exists $a \in \RR^d$ and a constant
$\gamma_a >0$, such that
$$\sum_{\cI \in \powerset \{1, \ldots, d \}} \int \tau_a(H_\cI) d\mu \geq \mu(\mathbbm{1}) + \gamma_a.$$
\end{prop}

The condition in the statement is obviously open with respect to $a$.

\section{Poisson transform}

Arguably the closest multivariate integral transform to the classical Cauchy transform in 1D is the Poisson transform. We elaborate below a few details, with \cite{Stein-Weiss} as a basic reference. 

Let $P(x,t) = c_d \left(\frac{t}{[t^2 + \| x \|^2]^{\frac{d+1}{2}}} \right)$ be Poisson's kernel in $\RR^d$. We denote 
$$P_\mu(x,t) = \int_{\RR^d} P(x-u,t) d\mu(u), \ \ x \in \RR^d, t>0,$$
the Poisson transform of the admissible measure $\mu$. It is a harmonic function in the upper-half space of $\RR^{d+1}$, which determines $\mu$ by non-tangential limits ($ t\mapsto 0$) in the distribution sense. Accordingly, given an admissible, indeterminate measure $\mu$ on $\RR^d$, there exists a value $(x_0,t_0)$ with the property that the continuous function $u \mapsto P(x_0-u,t_0)$ is separating with respect to $\mu$. In analogy to the 1D situation, we prove that almost any pair $(x_0,t_0) \in \RR^d \times (0,\infty)$ has this property. 
Performing an averaging in the mean of the respective monotonic approximation we obtain an everywhere criterion of indeterminateness.

\begin{thm} Let $\mu$ be an admissible measure on $\RR^d$. The following are equivalent:

a) The measure $\mu$ is moment indeterminate;

b) There exists $(x_0,t_0) \in \RR^d \times (0,\infty)$ such that
\begin{equation}\label{Poisson-barrier}
\kappa_\mu(x_0,t_0) := \inf_{ p(u) \leq P(x_0-u,t_0) \leq q(u)} \int (q-p) d\mu > 0,
\end{equation}
where $p(u), q(u)$ are polynomial functions,

c) For any point $(x_0,t_0) \in \RR^d \times (0,\infty)$ and radius $0< r< t_0$:
$$ \int_{S((x_0,t_0),r)} \kappa_\mu(x,t) d\sigma(x,t) >0,
$$ where $S((x_0,t_0), r)$ denotes the sphere and $d\sigma$ is its surface area measure.
\end{thm} 

\begin{proof} The equivalence between a) and b) follows from the fact that the Poisson transform $P_\mu$ determines $\mu$.
It remains to prove that a) implies c). Assume that the measure $\mu$ is indeterminate, that is there exists a moment equaivalent, admissible measure $\nu$, different than $\mu$. 

Suppose  that there exists $(x_0,t_0) \in \RR^d \times (0,\infty)$ such that
$$ \inf_{ p(u) \leq P(x_0-u,t_0) \leq q(u)} \int (q-p) d\mu = 0.$$ Then
$$ P_\mu(x_0,t_0) = P_\nu(x_0,t_0).$$
The harmonic function $u(x,t) = P_\mu(x,t)-P_\nu(x,t)$ is not identically zero, but vanishes at the point $(x_0,t_0)$. 
The zero set $V$ of $u(x,t)$ is a real analytic hypersurface of $\RR^d \times (0,\infty)$, which by the maximum principle cannot contain any euclidean sphere.

Note that the function $\kappa_\mu$ is Borel measurable. Its zero set is included in $V$, hence $\kappa_\mu >0$ on an open subset of any euclidean sphere, in particular  $S((x_0,t_0), r)$. So, its average on any sphere is non-zero if and only if the measure $\mu$ is indeterminate.
\end{proof}

\begin{cor} An admissible measure $\mu$ is indeterminate on $\RR^d$ if and only if
\begin{equation}
\int_{S((0,1), r)} \kappa_\mu(x,t) d\sigma(x,t) > 0
\end{equation} for some $0< r< 1$.
\end{cor} 

Given an admissible measure $\mu$ it is natural to consider the set of Poisson transforms of equivalent measures:
$$ \Delta_\mu(x,t) = \{ P_\nu(x,t): \ \ \nu \sim \mu\}.$$
This is a closed, convex set, hence an interval. According to the proof above,  $\Delta_\mu(x,t)$ reduces to a point if and only if $\kappa_\mu(x,t) =0$.

\begin{lem} The length of the interval $\Delta_\mu(x,t)$ is $\kappa_\mu(x,t)$.
\end{lem} 

\begin{proof} Riesz' extension theorem (see Theorem \ref{basic}) shows that we can populate the interval
$$  [\sup_{ p(u) \leq P(x-u,t)} \int p d\mu, \inf_{ q(u) \geq P(x-u,t)} \int q d\mu]$$
by values $\int P(x-u,t) d\nu(u)$, where $\nu$ is a measure, moment equivalent to the prescribed measure $\mu$.

\end{proof} 

We have just proved that either $\Delta_\mu(x,t)$ consists of a single point for every $(x,t) \in \RR^d \times (0,\infty)$, or this happens only on an exceptional locus contained in the zero set of a non-trivial harmonic function.

\begin{rem} Notes on the exceptional set $E= \{ (x,t), \ \kappa_\mu(x,t) =0 \}$ attached to an indeterminate, admissible measure $\mu$ on $\RR^d$.

If $\nu$ is a moment equivalent measure to $\mu$, then $E$ is contained in the zero set of the harmonic function $P_\mu-P_\nu$, in itself a real analytic subset of $\RR^d \times (0,\infty)$. On the other hand, if $\kappa_\mu(x,t) \neq 0$, then Riesz' extension theorem implies that there exists a moment equivalent measure $\nu$, such that $P_\mu(x,t) \neq P_\nu(x,t).$ In other terms,
$$ E = \bigcap_{\nu  \sim \mu} \{ (x,t): \ P_\mu(x,t) = P_\nu(x,t) \}.$$
But an arbitrary intersection or real analytic sets is real analytic, see for instance Cor. 2 on pg. 100 in \cite{Narasimhan}. All in all we have proved the following statement.

\begin{prop} Let $\mu$ be an indeterminate, admissible measure on $\RR^d$. Then the exceptional set $E= \{ (x,t), \ \kappa_\mu(x,t) =0 \}$ is (a proper, closed) real analytic subset of $\RR^d \times (0,\infty)$.
\end{prop} 

In particular, the complement of $E$ into $\RR^d \times (0,\infty)$ cannot have relatively compact connected components. Indeed, otherwise for any measure $\nu$, moment equivalent to $\mu$, the harmonic function $P_\mu(x,t) \neq P_\nu(x,t)$ would be identiucally zero on that component, and hence everywhere. A little more can be said in this respect.

\begin{prop} Let $\delta>0$ and let $H$ denote an irreducible, real analytic set of real dimension $d$, contained in $\RR^d \times (\delta,\infty)$. Let $E$ denote the exceptional set attached to an indeterminate measure on $\RR^d$. The germ of $E$ at any point cannot contain the germ of $H$.
\end{prop} 

\begin{proof} If the germ of $E$ contains that of $H$, then $E$ contains $H$ due to the irreducibility assumption. First we prove that
$$ \lim_{t^2 + \|x \|^2 \rightarrow \infty} P_\mu(x,t) = 0,$$ uniformly on the set $x \in \RR^d, t \geq \delta.$
Indeed, let $M$ be a positive number, so that
$$ \int_{\| u \| \geq M} \frac{t}{[t^2 + \| x -u \|^2]^{\frac{d+1}{2}}} d\mu(u) \leq [\frac{1}{t^{\frac{d+1}{2}}}] \int_{\| u \| \geq M} d\mu(u)$$
can be made arbitrary small whenever $t \geq \delta$ and $M$ tends to infinity.
On the other hand, if $t^2 + \|x\|^2 = R^2$, then
\begin{align*}
&\;\; \int_{\| u \| \leq  M} \frac{t}{[t^2 + \| x -u \|^2]^{\frac{d+1}{2}}} d\mu(u) \\
\leq& \;   \int_{\| u \| \leq  M} \frac{R}{[t^2 + \|x\|^2 - 2 \langle x, u \rangle + \|u \|^2]^{\frac{d+1}{2}}}d\mu(u) \\\
\leq& \;
\int_{\| u \| \leq  M} \frac{R}{[R^2 - 2 R M]^{\frac{d+1}{2}}} d\mu(u)
\end{align*}
and the latter integral converges uniformly to zero for $R \rightarrow \infty$.

Assume $H \subset E$ and let $\nu$ be another admissible measure, moment equivalent to $\mu$. For $R>0$ large, the complement of $H \cup S((0,0), R)$ contains a connected component, relatively compact in $\RR^d \times (\delta,\infty)$. The harmonic function $P_\mu-P_\nu$ vanishes on $E$, hence on $H$, and it is arbitary small on the sphere. The maximum principle implies $P_\mu = P_\nu$ everywhere, a contradiction.
\end{proof}

\end{rem}

\section{Preservers of indeterminate measures}

\subsection{Push forward by projections}
Both Fourier transform and Fanttapi\`e transform criteria, respectively Theorems \ref{Fourier} and \ref{CONE}, contain a separating function
of the form $x \mapsto h(\xi \cdot x)$ associated to a privileged vector $\xi \in \RR^d$.  Consequently, an orthogonal push forward measure on linear varieties $V$ containing $\xi$ will preserve the indeterminateness.

Indeed, let $\pi: \RR^d \longrightarrow V$ be the orthogonal projection and assume $\mu$ is an admissible measure on $\RR^2$, respectively supported by a closed cone
$\Gamma^\ast$. Assume, on the respective supports, and with running polynomials $p$ and $q$:
\begin{equation}\label{projection}
 \sup_{p(x) \leq h(\xi \cdot x)} \int p d\mu < \inf_{q(x) \geq  h(\xi \cdot x)} \int q d\mu.
 \end{equation}
The function $h(\xi \cdot x)$ is constant along the fibers of $\pi$, hence $h = \pi^\ast g$, where $g: V \longrightarrow \RR$ is a polynomial function. Denoting by $r,s$ polynomials on $V$ we infer:
$$ \sup_{r \leq g} \int r d\pi_\ast \mu < \inf_{s \geq g} \int s d \pi_\ast \mu,$$
or, equivalently,
$$ \sup_{ r(\pi(x)) \leq h(\xi \cdot x)} \int r(\pi(x)) d\mu <   \inf_{ s(\pi(x)) \leq h(\xi \cdot x)} \int s(\pi(x)) d\mu$$
which is true in view of (\ref{projection}).

Both Fourier transform and Fantappi\`e transform criteria, respectively Theorems \ref{Fourier} and \ref{CONE}, contain a separating function
of the Radon transform form $x \mapsto h(\xi \cdot x)$ associated to a privileged vector $\xi \in \RR^d$.  Consequently, an orthogonal push forward measure on linear varieties $V$ containing $\xi$ will preserve the indeterminateness.

A celebrated theorem of Petersen \cite{Petersen} relates multivariate moment determinateness to the same property of the marginals. The proof was obtained via a natural weighted approximation scheme.
Remarking the invariance under linear changes of coordinates, we offer the following partial complementary picture.

\begin{prop} 
\label{prop:PETERSENCOMP}
Let $\mu$ be an admissible positive measure.  Then the following statements hold{\rm :}
\begin{enumerate}
\item[{\rm (a).}] If there exists a basis of $\RR^d$ with the property that the $1D$ marginals of $\mu$
along parallel projections with respect to this basis are all determinate, then $\mu$ is determinate.
\bigskip

\item[{\rm (b).}] Assuming the measure $\mu$ is indeterminate, there exists a coordinate, possibly non-orthogonal, frame with the property
that all marginals of $\mu$ are indeterminate.
\end{enumerate}
\end{prop}

\begin{proof} Part a) is Petersen's theorem, translated to an arbitrary linear basis.
To prove part b) we remark that the unit vectors along which the projection of $\mu$ is indeterminate form an open set on the sphere.
\end{proof}

\subsection{Small perturbations} Both separating functions $e^{i \xi \cdot x}$ and $\frac{1}{1+ a \cdot x}$ appearing the the previous section
are uniformly bounded on the respective supporting sets. As a matter of fact, a bounded separating function always exists. This simple remark shows that
a perturbation of admissible measures, small in total variation, preserves indeterminateness.

\begin{lem} Let $X \subset \RR^d$ be a closed subset. The space $C_0(X)$ of continuous functions of compact support contains separating functions
for any admissible, indeterminate measure supported by $X$.
\end{lem} 

\begin{proof} Let $\mu, \nu$ be two admissible measures on $X$, moment equivalent, but distinct. That is, there exists a function $\phi \in C_p(X)$
with the property $\int \phi d\mu \neq \int \phi d\nu$. Let $K_n \subset {\rm int} K_{n+1}$ be an exhaustion of $X$ with compact sets, with an attached 
system of continuous functions $\kappa_n, \ 0 \leq \kappa_n \leq 1,$ satisfying
$$ {\rm supp} (\kappa_n) \subset K_{n+1}, \ \ \kappa_n(x) = 1, \ x \in K_n.$$
Lebesgue dominated convergence theorem implies
$$ \lim_n \int \kappa_n \phi d\mu = \int \phi d\mu,$$
and similarly for the measure $\nu$. That is, for $n$ sufficiently large, the functions $\kappa_n f$ are separating the measures $\mu$ and $\nu$.
\end{proof} 

In particular, on any support and any indeterminate measure there are uniformly bounded separating functions.

\begin{prop} Let $X \subset \RR^d$ be a closed set and let $\mu$ be an admissible, $X$-indeterminate measure. There exists $\epsilon >0$
with the property that for all admissible measures $\sigma$ supported on $X$, of total mass less than $\epsilon$, the measure $\mu + \sigma$
is admissible and $X$-indeterminate.
\end{prop}

\begin{proof} According to the Lemma above, if $\nu$ is an admissible measure on $X$, moment equivalent to $\mu$, there exists a separating function
of compact support $\phi \in C_0(X)$. We can assume $\| \phi \|_{\infty} =1$. Choose
$$ \epsilon < | \int \phi d\mu - \int \phi d\nu|.$$
 The measure $\sigma$ in the statement has the property that
$\mu + \sigma$ and $\nu + \sigma$ are moment equivalent, yet 
$$ |\int \phi d(\mu + \sigma) - \int \phi d(\nu+ \sigma)| > 0.$$
\end{proof}

\subsection{Convolution}  Convolution transforms change in general the moment data,  preserving however the indeterminateness feature, and possibly
improving the regularity of the original measure.

\begin{prop} Let $\mu$ be an indeterminate, admissible measure on $\RR^d$ and let $\tau \neq 0$ be a positive measure of compact support.
Then the convolution $\tau \ast \mu$ is admissible and indeterminate.
\end{prop}

\begin{proof} Since $\tau$ has compact support, the convolution $\tau \ast \mu$ is well defined in the sense of distributions, and it is a measure.
Let $p(x)$ be a polynomial. By definition,
$$ \int p d (\tau \ast \mu) = \int \int p(x+y) d\tau(y) d\mu(x). $$
The iterated integral is convergent due to the assumption that the measure $\mu$ is admissible, that is all continuous functions of polynomial growth are
$\mu$-integrable. 

Taylor's expansion yields the finite sum
$$ p(x+y) = \sum_k \frac{p^{(k)}(y)}{k!} x^k.$$
We deduce that, if $\nu$ is an admissible measure, moment equivalent to $\mu$, then
$$ \int \int p(x+y) d\tau(y) d\mu(x) = \int \int p(x+y) d\tau(y) d\nu(x). $$
And that is true for all polynomials $p$.

Assume that the two moment equivalent measures $\mu$ and $\nu$ are distinct. We have to prove that $\tau \ast \mu$ is different than
$\tau \ast \nu$. Passing to Fourier transforms we find
$$ \widehat{\tau \ast \mu} = \hat{\tau} \hat{\mu}.$$
But $\hat{\tau}$ is an entire function on $\CC^d$, while $\hat{\mu}$ is a function of class ${\mathcal C}^{\infty}(\RR^d)$.
Since $\tau \neq 0$, the zeros of the Fourier transform $\hat{\tau}$ are supported on an analytic hypersurface in $\CC^d$, hence they are nowhere dense
in $\RR^d$.

Assuming by contradiction $ \hat{\tau} \hat{\mu} =  \hat{\tau} \hat{\nu}$ we find $\hat{\mu} = \hat{\nu}$, that is $\mu = \nu$.
\end{proof}

The preceding proof can be adapted to more general measures $\tau$, assuring the convergence of the convolution integral. Letting $\tau = \phi dx$, where $\phi$ is a continuous function of compact support one can produce new indeterminate measures $\phi \ast \mu$ which are absolutely continuous with respect to Lebesgue measure. The well known support inclusion
$$ {\rm supp} (\tau \ast \mu) \subset {\rm supp} \tau \ + \ {\rm supp} \mu$$
indicates how to adapt the convolution transform to prescribed supports.

A natural convolution transform of an admissible measure $\mu$ is its Newtonian potential
$ U^\mu(x) = \int E(x-y) d\mu(y)$, where $E(x) = - \log \| x \|$ in dimension $d=2$ and $E(x) = \|x\|^{2-d}$ for $d>2$. The fine  properties of the 
functions $U^\mu(x)$ are well studied, to illustrate only by a very recent contribution \cite{Verdera}. Since $\Delta U^\mu = {\rm const.} \mu$, one can also regard translates $E(x-y)$ of the fundamental solutions as separating functions for measures. We do not pursue this path here.

\bigskip

\subsection{Positive weights} Multiplying an indeterminate measure with well adapted polynomial weights does not transform it into a determinate measure. We present such an observation.

\begin{prop}\label{multiplier} Let $X \subset \RR^d$ be a closed set and let $\mu$ be an $X$-indeterminate measure. If $w$ is a non-negative polynomial on $X$, with finitely many zeros on $X$, then the measure $w\mu$ is still indeterminate.
\end{prop}

\begin{proof} Let $\nu$ be an admissible measure supported on $X$, moment equivalent to $\mu$, but different than $\mu$. Therefore there exists a continuous function $\phi$ on $X$,
of polynomial growth, such that $\int \phi d\mu \neq \int \phi d\nu$. Let $q$ be a polynomial which interpolates $\phi$ on the zero set $V$ of $w$. Since $\int q d\mu = \int q d\nu$ we infer
$$ \int (\phi - q) d\mu \neq \int (\phi-q) d\nu.$$
Remark that he indeterminate measures $\mu, \nu$ are not supported on the finite set $V$. 

Choose for every $\delta>0$ a continuous function $\chi_\delta$ with the properties $0 \leq \chi_\delta \leq 1$,
$\chi_\delta (x) = 0, \ \ {\rm dist}(x,V) < \delta$ and $\chi_\delta(x) = 1, \ \ {\rm dist}(x,V) > 2\delta$. In virtue of Lebesgue dominated convergence theorem, there exists $\epsilon > 0$
satisfying
$$ \int \chi_\epsilon (\phi - q) d\mu \neq \int \chi_\epsilon (\phi-q) d\nu.$$
For $R>0$ the Baire-1 function 
$$ \tilde{w}_R(x) = \left\{ \begin{array}{c} w(x), \ |x|<R,\\
                                                        \max (1, w(x)), \  |x|\geq R, \end{array} \right. $$
                                                        has the property $\frac{w}{\tilde{w}_R}(x) \leq 1, \ \ |x| \geq R.$
                                                        
The measure $\mu$ decays fast at infinity. In particular, for every continuous function $\xi$ of polynomial growth one finds
$$ \lim_{R \rightarrow \infty} \int \frac{w}{\tilde{w}_R} \xi d\mu = \int \xi d\mu,$$
and similarly for $\nu$.

Then the Baire-1 function $\Phi = \frac{ \chi_\epsilon (\phi - q) }{\tilde{w}_R}$ has polynomial growth and for $R$ large enough
$ \int \Phi w d\mu \neq \int \Phi w d \nu.$ On the other hand, the measures $w\mu$ and $w \nu$ are moment equivalent.
\end{proof} 

A typical choice of weight, for a prescribed finite set $V$, is $w(x) = \prod_{\lambda \in V} \| x-\lambda \|^2.$
\bigskip

\subsection{Equivariant setting} Let $X \subset \RR^d$ be a semialgebraic set and let $G$ be a compact group acting on $X$. The Haar measure on $G$ is denoted $dH$.

\begin{lem}
If two $G$-invariant, admissible measures on $X$ are moment equivalent, but distinct, then there exists a $G$-invariant separating function.
\end{lem}

\begin{proof} Let $\phi \in C(X)$ be a continuous separating function for the two measures. The invariance of the measure to translations by an element $g$ of $G$ yields
$$ \int \phi (gx) d\mu(x)  = \int \phi(x) d\mu(x) \neq \int \phi(x) d\nu(x)  = \int \phi (gx) d\nu(x).$$
Consequently the $G$-invariant average $(\phi^G)(x) = \int \phi(gx) dH(g)$ satisfies
$$\int \phi^G d\mu = \int \phi d\mu \neq \int \phi d\nu = \int \phi^G d\nu.$$
\end{proof} 

The above lemma shows that the push-forward measures on the quotient space $X/G$ are distinct. In case $X/G$ can be identified with a semi-algebraic set of Euclidean space,
and the projection map is induced by polynomials, then the two push-forward measures are moment equivalent, but distinct. This is the case of the action of the rotation group on $\RR^d$. A detailed analysis of the latter scenario was carried out by Berg and Thill  \cite{BT}.

\subsection{Completely monotonic functions}

The integral transforms we invoked in producing separating functions for indeterminate measures are derived from completely monotonic functions of a single variable.
We provide below a general framework of producing lower and upper polynomial functions (as a matter of fact MacLaurin polynomials) of completely monotonic functions.
Although the determinateness criteria obtained this way may not be sharp, they open a natural path towards a deeper study of multivariate moment problems. 

A smooth function $\phi : [0, \infty) \longrightarrow \RR$ is called {\it completely monotonic} if
$$ (-1)^n \phi^{(n)}(x) \geq 0,  \ \ x \geq 0.$$
A classical characterization due to Bernstein relates these functions to Laplace transforms of positive measures on the semi-axis, \cite{HenkinShananin}.
The MacLaurin polynomial of degree $n$ of a smooth function $\phi$ is denoted, in short, by
$$ M_n(\phi)(x) = \sum_{j=0}^n \frac{ \phi^{(n)}(0)}{n!} x^n.$$

\begin{lem} Let $\phi : [0, \infty) \longrightarrow \RR$ be a completely monotonic function. Then
$$ M_{2n-1}(\phi)(x) \leq \phi(x) \leq M_{2n}(\phi)(x), \ \ n \geq 1, x \geq 0.$$
\end{lem} 

\begin{proof}
Taylor's formula with a remainder in integral form 
$$ \phi(x) = M_n(\phi)(x) + \int_0^x \frac{(x-t)^n}{n!} \phi^{(n+1)}(t) dt, \ \ n \geq 0, x \geq 0,$$
implies the result.
\end{proof}

Let $\omega : \RR^d \longrightarrow [0,\infty)$ be a continuous weight. Passing now to several variables and assuming $\phi(\omega (x))$ is a separating function for an indeterminate admissible measure $\mu$ on $\RR^d$, we infer
$$ \inf_n \int [M_{2n}(\phi)(\omega(x)) - M_{2n-1}(\phi)(\omega(x))] d\mu(x) > 0,$$
or, equivalently,
$$ \inf_n  \left(\frac{ \phi^{(2n)}(0)}{(2n)!} \int \omega(x)^{2n} d\mu(x) \right) > 0.$$

Similarly, the monotonic approximation of the function $\cos t$ by its MacLaurin polynomials led Riesz to an effective 1D moment indeterminateness condition, see pg. 47 of
\cite{Riesz-3}. An alternate proof with some variations of the same criterion appears in Freud's monograph \cite{Freud} section II.5. Both criteria are however slightly weaker than the well known Carleman condition. Regardless to say that, in the multivariate setting, Riesz' determinateness criterion imposed on the marginals of a measure assure its joint determinateness.

\section{Bounded point evaluations}

The guiding light of univariate moment indeterminateness is the existence of bounded point evaluations at non-real values. This is well encoded in the 
non-vanishing of the Christoffel function associated to a moment data. Much less is known in the multivariable setting. Below we derive a few existence results of bounded point evaluations, or rather bounded hyperplane evaluations, from the observations contained in the previous section.

As before, let  $\Gamma \subseteq \RR^d$ be an acute, convex and solid cone and let $\Gamma^\ast$ be the dual cone.
For a point $a \in \RR^d$ we denote
$$H_a = \{ x \in \RR^d: a \cdot x + 1 =0\}.$$
Notice that a polynomial $R \in \RR[x_1, \ldots, x_d]$ vanishes on $H_a$ if and only if $R$ factors through $a \cdot x + 1$.

\begin{prop}
If an admissible measure $\mu$ supported on $\Gamma^*$ is moment indeterminate on $\Gamma^*$, then there exists $a \in {\rm int}\, \Gamma$ such that at least one of the following quantities is positive, i.e., 
$$
\inf \left \{ \int_{\Gamma^*} r \, d\mu: r|_{H_a} = 1 \; {\rm and}\; r|_{\Gamma^*} \geq 0 \right\} > 0
$$
or
$$
\inf \left \{ \int_{\Gamma^*} r \, d\mu: r|_{H_a} = -1 \; {\rm and}\; r|_{\Gamma^*} \geq 0 \right\} > 0,
$$
where $r$ is a polynomial.
\end{prop}

\begin{proof}
Notice that $r|_{H_a} = 1$ implies that $r(x) = 1 - (a \cdot x + 1)p(x)$ for some $p \in \RR[x]$ and
$$\int [ 1 - (a \cdot x + 1)p(x) ] d\mu \geq \int \left[ \frac{1}{a \cdot x + 1} - p(x) \right] d\mu.$$
Similarly, if $r(x) = -1+ (a \cdot x +1) q(x)$, then
$$\int [ -1 + (a \cdot x + 1)q(x) ] d\mu \geq \int \left[ \frac{-1}{a \cdot x + 1} + q(x) \right] d\mu.$$
The necessity of the positivity of at least one of the two infima now follows immediately from Theorem \ref{CONE}. 
\end{proof}

In the same spirit, allowing now a positive weight against the indeterminate measure, one finds a necessary and sufficient condition.

\begin{thm}
let  $\Gamma \subseteq \RR^d$ be an acute, convex and solid cone and let $\Gamma^\ast$ be the dual cone. Let $\mu$ be an admissible measure supported on $\Gamma^\ast$. The measure $\mu$ is moment indeterminate with respect to the support $\Gamma^\ast$ if and only if there exists 
$a \in {\rm int} \Gamma$ with the property that at least one of the conditions
$$ \inf \{ \int_{\Gamma^\ast} \frac{ r(x) d\mu(x)}{1+ \|x \|} : \ \ r|_{H_a} = \pm 1, \ r|_{\Gamma^\ast} \geq 0, \ r \in \RR[x] \} > 0$$
is satisfied.
\end{thm}

\begin{proof}
If $a$ is an interior point of the convex cone $\Gamma$, then the angle between $a$ and any element of the dual cone is bounded from above by a
constant less than $\pi/2$. Consequently, there exists $\gamma>0$ with the property:
$$ a \cdot x \geq \gamma \| x \|, \ \ x \in \Gamma^\ast.$$
Then there are positive constants $C_1, C_2$ so that
$$ \frac{C_1}{1+ \| x \|} \geq \frac{1}{a \cdot x + 1} \geq  \frac{C_2}{1+ \| x \|}, \ \ x \in \Gamma^\ast.$$
In, particular, any polynomial $r$ which is non-negative on $\Gamma^\ast$ satisfies
$$  C_1 \frac{r(x)}{1+ \| x \|} \geq \frac{r(x)}{a \cdot x + 1} \geq  C_2 \frac{r(x)}{1+ \| x \|}, \ \ x \in \Gamma^\ast.$$
Theorem \ref{CONE} completes then the proof.
\end{proof}
\bigskip

A natural base change of rational curves provides a generalization of the moment determinateness in terms of the existence of bounded point evaluations outside the real locus. Specifically, we focus on a real, affine, algebraic curve $X \subset \RR^d$ which admits a proper parametrization
$$ u : \RR \longrightarrow X,$$
where $u$ is a map consisting of real polynomials, such that $u(\RR)$ omits at most finitely many points of $X$ and which is one-to one except finitely many points of $\RR$. We call such an object a {\it real, polynomial curve}. We refer the reader to Chapter 7 of \cite{SWP}, which is fully dedicated to real parametrizations of real, rational curves. The algebra of regular functions on $X$ is denoted $\RR[X] = \RR[x_1,x_2,\ldots,x_d]/I_X$, where $I_X$ is the reduced ideal associated to the variety $X$.
The complexification of $X$ is denoted $X_\CC$, as defined by the ideal $I_X \otimes_\RR \CC \subset \CC[x]$. The polynomial parametrization $u$
extends to $U: \CC \longrightarrow X_\CC$, which remains a proper map with finite fibres. Moreover, the curves $\CC$ and $X_\CC$ are birationally equivalent, that is the pull-back $U^\ast: \CC(X_\CC) \longrightarrow \CC(t)$ is an isomorphism of algebras of rational functions.

We recall the definition of a basic concept in approximation theory. An admissible measure $\nu$ defined on an affine, complex algebraic variety $X_\CC$
admits {\it analytic bounded point evaluations} if there exists a non-trivial open subset $V \subset X_\CC$ and a positive constant $C$, with the property
$$ |p(\lambda)| \leq C \| p \|_{2,\nu}, \ \ \lambda \in V, \ p \in \CC[X_\CC].$$
One can replace the point evaluations by a Bergman or Hardy space norm. For instance, a prescribed point $\lambda \in V$ and a radius $r>0$ such that
the disk $D = D(\lambda,r)$ is fully contained with its closure in $V$ give rise to an estimate of the form:
$$ \| p \|_{2,D} \leq C' \| p \|_{2,\mu},$$
where $C'$ is a positive constant. And vice-versa, such a Bergman space inequality assures the existence of analytic bounded point evaluations inside the disk $D$.

The above estimates refer to polynomial functions, but they can be replaced for instance by rational function or analytic functions subject to growth conditions.
A landmark result of J. Thomson establishes the existence of analytic bounded point evaluations for compactly supported measures on $\CC$ in terms of
the non-density of complex polynomials in Lebesgue space \cite{Thomson}. In our context, M. Riesz' theorem (see Section 22 in \cite{Riesz-3}) asserts that an admissible measure  supported by the real line is indeterminate if and only if it admits analytic bounded point evaluations in the complex plane. We aim at extending this phenomenon to real polynomial curves.

\begin{thm} Let $X \subset \RR^d$ be a real, polynomial curve and let $\mu$ be an admissible, positive measure supported on $X$.
The measure $\mu$ is $X$-indeterminate if and only if there exist bounded analytic point evaluations for $\mu$, supported on the complexification $X_\CC$.
\end{thm}

\begin{proof}

Let $ u : \RR \longrightarrow X$ denote a proper, polynomial parametrization of the affine curve $X$. Let $F \subset X_\CC$ be a finite set with the property
that the restricted maps
$$ u: \RR \setminus u^{-1}(F) \longrightarrow X \setminus F$$
and
$$ U: \CC \setminus U^{-1}(F) \longrightarrow X_\CC \setminus F$$
are bijective. The minimal set $F$ with this property will be called by abuse of terminology {\it the ramification locus} of the curve $X$. Note that, in general $F$ depends on the parametrization.

Denote $G = U^{-1}F \subset \CC$. Any polynomial $h \in \CC[t]$ vanishing on the finite set $G$ is of the form
$h = U^\ast q = q \circ U,$ where $q \in \CC[X_\CC]$ and $q|_F = 0$. 
Indeed, there exists a rational function $r \in \CC(X_\CC)$ with the property $h = U^\ast r$. The local structure of the finite map $U$ implies that
$r$ is a regular function at all points of $X_\CC$. Lemma 2.1 in \cite{Harris} implies $r \in \CC[X_\CC]$. If the polynomial $h$ is real, and vanishes on 
the finite set $G$, then one can choose a real element $q \in \RR[X]$ satisfying $h = U^\ast q = u^\ast q$. 
Denote by $w \in \RR[t]$ the polynomial with simple zeros at $G$. Accordingly, $w = U^\ast \omega$, where $\omega \in \RR[X]$.
We note in passing that the class of smooth, complete intersection curves allowing a polynomial parametrization is distinguished by an empty ramification locus, \cite{BenbouzianeHouariKahoui}. 

 Assume two admissible measures $\mu_1, \mu_2$ supported by $X$ on are moment equivalent, but distinct. According to Proposition \ref{multiplier} the measures $\omega^2 \mu_1$
 and $\omega^2 \mu_2$ are moment equivalent and distinct, and similarly for $\omega^4 \mu_j, \ j =1,2.$

Since the restricted map $u|_{\RR \setminus u^{-1}F}$ is bijective and the measures $\omega^2 \mu_1, \omega^2 \mu_2$ do not possess atoms on $F$, one can define by push-forward positive measures $\sigma_1, \sigma_2$
on $\RR \setminus u^{-1}F$ with the properties
$$ \int h(u(t)) \sigma_j(dt) = \int_{X\setminus F} h \omega^2 d\mu_j, \ \ h \in \RR[X], \ j=1,2.$$
\bigskip

One step further, one considers the measures $w^2 \sigma_j$ and their push-forward by $u$ measures $\omega^4 \mu_j, \ j=1,2.$
Since every element $h \in \RR[t]$ ``descends" to $X$ after multiplication by $w^2$:  $ hw^2 = u^\ast q$, with $q \in \RR[X]$, we infer
that the measures $ \sigma_1, \sigma_2$ are moment equivalent. Indeed, for every $h \in \RR[t]$ one finds:
$$ \int h(t) w^2(t) \sigma_1 (dt) = \int q(u(t)) \sigma_1(dt) = \int q \omega^2 d\mu_1 =  $$ $$ 
\int q \omega^2 d\mu_2 =  \int q(u(t)) \sigma_1(dt) = \int h(t) w^2(t) \sigma_1 (dt).$$

Assume by contradiction that $w^2 \sigma_1 = w^2 \sigma_2$ as Radon measures on the line. Then $\sigma_1-\sigma_2$ is a sum of point masses
concentrated on the set $G \cap \RR$. By construction, $\sigma_1, \sigma_2$ do not have point masses on $G \cap \RR$, hence $\sigma_1 = \sigma_2$.
Therefore $\omega^2 \mu_1 = \omega^2 \mu_2$, a contradiction. 

In view of M. Riesz' theorem there exist analytic bounded point evaluations on (relatively compact subsets of) $\CC \setminus \RR$, for $w^2\sigma_1, w^2\sigma_2$. As a matter of fact every point
$\alpha \in \CC \setminus \RR$ satisfies
\begin{equation}\label{weight}
 |p(\alpha)|^2 \leq C \int p^2 w^2 d\sigma_j, \ \ j=1,2,
 \end{equation}
where $C = C(\alpha) >0$ is locally bounded and $p \in \RR[t].$
On the base $X$ of the map $u$ we infer:
$$
|f(\beta)|^2 \leq C(\beta)  \int f^2 \omega^4 d\mu_1 = C(\beta) \int f^2 \omega^4 d\mu_2, \ f \in \RR[X],
$$ where $\beta = U(\alpha) \in X_\CC \setminus X$. The constant $C(\beta)$ is still locally bounded as a function of $\beta$.
Choose a point $\lambda \in X_\CC \setminus X$ and a sufficiently small radius $r$, so that the function $|w|$ does not vanish on the closure of the disk $D = D(\lambda,r)$. Consequently there are constants $M,M' >0$ with the property:
$$ \| f \|_{2,D} \leq M'\left \| \frac{f}{|w|^2} \right\|_{2,D} \leq M \left\|  \frac{f}{|w|^2} \right\|_{2, w^4 \mu} = M \| f \|_{2,\mu}, \ \ f \in \RR[X].$$

To prove the other implication we assume that the admissible measure $\mu$ defined on $X$ admits analytic bounded point evaluations. If the measure $\mu$ does not possess point masses on the ramification locus $F$, then one can define without ambiguity an admissible measure $\sigma$ on $\RR$, with the property
$\mu = u_\ast \sigma$. Even if the measure $\mu$ has atoms on $F$ one can choose, with some degree of freedom, such a lift $\sigma$ by selecting point masses on $G = u^{-1} F$ which satisfy
$$ \int_\RR f\circ u d\sigma = \int_X f d\mu, \ \ f \in \RR[X].$$
The assumption on the existence of analytic bounded point evaluations imposed on the measure $\mu$ implies that there exist $\lambda \in \CC \setminus \RR$ and $r>0$ carrying the estimate
\begin{equation}\label{Bergman}
 \| h \|_{2, D} \leq M \| h \|_{2, \sigma}, \ \ h \in V,
 \end{equation}
where $V = I_G \subset \CC[t]$ is a finite codimensional subspace.

Next we prove that a bound of the form (\ref{Bergman}) is true for all polynomials $h \in \CC[t]$. To this aim, let $h_1, \ldots, h_n \in \CC[t]$ be a linearly independent system
which spans $\CC[t]/V$. Denote by $P^2(\mu)$ the closure of polynomials in $L^2(\mu)$ and respectively by $L_a^2(D)$ the closure of polynomials in $L^2(D, {\rm d \ Area})$.
The entire algebra $\CC[t]$, and a fortiori $V$, is a subspace of both $P^2(\mu)$ and $L^2_a(D)$. Denote by $H$ the closure of $V$ in $P^2(\mu)$ and by $K$ the closure of $V$ in
Bergman space $L^2_a(D)$. The restriction map
$$ R: H \longrightarrow K, \ \ R(h) = h|_D, \ \ h \in V,$$
is linear and continuous by \eqref{Bergman}. Denote by $P_H$ the orthogonal projection of $P^2$ onto $H$. The elements $f_j - P_H(f_j), \ 1 \leq j \leq n,$ span the finite dimensional orthogonal complement $H_1 = P^2(\mu) \ominus H$.

In particular every polynomial $h \in \CC[t]$ can be decomposed as 
$$ h = h_1 + h_2, \ \ h_1 \in H_1, h_2 \in H.$$
The restriction of polynomials to the disk $D$ extends by continuity to the operation 
$$ (f_j - P_H(f_j))|_D := (f_j)|_D - R P_H(f_j), \ \ 1 \leq j \leq n,$$
hence one can define 
$$h|_D = h_1|_D + R (h_2)$$ to the effect
$$ \| h\|^2_{2,\mu} = \| h_1\|^2_{2,\mu} + \|h_2\|^2_{2,\mu}.$$
On the other hand, $H_1$ is a finite dimensional space, endowed with a continuous restriction map to the Bergman space $L^2_a(D)$. Hence there exists a constant $M_1$ with the property.
$$ \| h_1 \|_{2,D} \leq M_1 \| h_1 \|_{2,\mu},\ h_1 \in H_1.$$
In conclusion, for all polynomials $h \in \CC[t]$ an estimate  (\ref{Bergman}) exists. According to M. Riesz' theorem, the measure $\sigma$ is indeterminate.

Let $\tilde{\sigma}$ be an admissible measure on the real line, moment equivalent to $\sigma$,  but different from $\sigma$. Clearly the measure $\tilde{\mu} = u_\ast \tilde{\sigma}$ is
moment equivalent to $\mu$. If $\mu = \tilde{\mu}$, then the measures $\sigma, \tilde{\sigma}$ coincide on the space of continuous functions on $\RR$, vanishing on the finite fibre $G$.
The proof of Proposition \ref{multiplier} implies $\sigma = \tilde{\sigma}$, a contradiction.
\end{proof}

One may naturally ask what happens on rational curves. The first observation is that push forward via rational parametrizations may alter indeterminateness. A simple example is the 
embedding
$$ F(t) = \left(t, \frac{1}{1+t^2} \right), \  t \in \RR.$$
Any admissible measure in $\RR^2$ supported by the image curve $\Gamma$ of equation $ y(1+x^2) =1$ is determinate, due to the fact that bounded polynomials on $\Gamma$ are separating points.
\bigskip

The class of real algebraic curves on which a positive polynomial function is a sum of squares was thoroughly studied by Plaumann \cite{Plaumann} and Scheiderer \cite{Scheiderer1}, \cite{Scheiderer2}, \cite{Scheiderer3} and \cite{Scheiderer4}. Similarly to the real line, on such curves, the mere existence of bounded point evaluations has moment determinateness implications.

\begin{thm} Let $X \subset \RR^d$ be a real algebraic curve on which all positive polynomials are sums of squares. Let $\mu$ be an admissible measure supported by $X$
which admits analytic bounded point evaluations on $X_\CC$. Then the measure $\mu$ is indeterminate.
\end{thm} 

\begin{proof} Assume first that $\beta \in X_\CC \setminus X$ is a $P^2(\mu)$-bounded point evaluation. That is:
$$ |p(\beta)|^2 \leq C \| p \|_{2,\mu}^2, \ \ p \in \RR[X].$$
Note that
$$\inf_{x \in X}  \| x -\beta \| \geq \delta > 0.$$
We prove that the function $\frac{1}{\| x-\beta \|^2}$ is separating for the measure $\|x-\beta\|^2 d\mu(x)$.

Indeed, let $q \in \RR[X]$ satisfy 
$$ \frac{1}{\| x-\beta \|^2} - q(x) > 0, \ \ x \in X.$$
By assumption, the positive polynomial $1- \|x-\beta\|^2 q(x)$ is a finite sum of squares of elements $p_k \in \RR[X]$:
$$ 1- \|x-\beta\|^2 q(x) = \sum_{k=1}^m p_k(x)^2.$$
Note that $ \sum_{k=1}^m |p_k(\beta)|^2 =1$. Consequently,
$$ \sum_{k=1}^m \int |p_k(x)|^2 d \mu(x) \geq \frac{1}{C} \sum_{k=1}^m |p_k(\beta)|^2 = \frac{1}{C} > 0.$$
Therefore
$$ \inf_{q(x) <  \frac{1}{\| x-\beta \|^2}} \int [\frac{1}{\| x-\beta \|^2} - q(x)] \| x-\beta\|^2 d\mu(x) >0,$$
that is the continuous function of polynomial growth $\frac{1}{\| x-\beta \|^2}$ is separating for the measure $\|x-\beta\|^2 d\mu(x)$.

The assumption in the statement was stronger. In particular, we can assume that a full neighborhood of $\beta$ consists of $P^2(\mu)$ bounded point evaluations
with respect to the same constant. For instance
$$ |p(\alpha)|^2 ||\alpha-\beta\|^4 = | p(\alpha) \| \alpha-\beta\|^2|^2 \leq C \int |p(x)|^2 ||x-\beta\|^4 \frac{\mu( d x)}{||x-\beta\|^4}, \ \ p \in \CC[X].$$
By choosing $\alpha$ in an open ball $B = B(\lambda, r)$ with $\| \beta - \alpha \| > r$ we infer
$$ \int_B |p(z)|^2 d {\rm vol} \leq M \int_X |p(x)|^2 \| x - \beta \|^4 \frac{\mu( d x)}{||x-\beta\|^4}, \ \ p \in \CC[z].$$
That is, the measure $\frac{\mu}{||x-\beta\|^4}$ carries analytic bounded point evaluations on a finite codimensional subspace of $P^2(\mu)$ (and of Bergman's space
$L^2_a(B)$). Indeed, $X_\CC$ is an algebraic curve and the ideal generated by the function $\|z-\beta\|^2$ in $\CC[X_\CC]$ has finite codimension.
The proof of the preceding theorem implies that $\frac{\mu}{||x-\beta\|^4}$ admits analytic bounded point evaluations.

Then the first part of the proof takes over, implying that the measure $\tilde{\mu} = \| x-\beta\|^2 \frac{\mu}{||x-\beta\|^4}$ is indeterminate. Finally, Proposition \ref{multiplier} completes the proof.
\end{proof}

\section{Examples}

\subsection{Finite maps of real algebraic varieties}

It is well known that there exists an admissible measure $\sigma$ supported on $[0,\infty)$ which is $\RR$-indeterminate, but $[0,\infty)$-determinate. The stark difference between Stieltjes, respectively Hamburger, determinateness is well analyzed in Berg and Thill \cite{BT}.

We put this pathology on algebraic curves, showing that even finite morphisms of curves are not expected to preserve moment indeterminateness. More precisely, let $\Gamma$ be the parabola of equation $x = y^2$ in $\RR^2$ and define the measure $\mu$, supported on $\pi$ as follows:
$$ \int f(x,y) d\mu = \int_\RR f(t^2,t) d\sigma(t), \ f \in \RR[x,y].$$
The projection of the $y$-axis is an isomorphim of smooth curves, therefore the measure $\mu$ is $\Gamma$-indeterminate. On the other hand, consider the projection $\pi$ of $\RR^2$ onto the $x$-axis. The image $\pi \Gamma = [0,\infty) $ of the curve $\Gamma$ is not a smooth variety, but only, as expected over the real field, simply a semi-algebraic set.
The measure $\pi_\ast \mu$ acts as follows:
$$ \int h d \pi_\ast \mu = \int h(t^2) d \sigma(t), \ \ h \in \RR[t].$$
But this measure is moment-determined on the line.

It is worth mentioning that the above exceptional measure $\mu$ has a point mass on the ramification locus of the projection map $\pi$.

\subsection{Real Polynomial Curves}

The affine curves allowing a polynomial parametrization referred to in our study were known forever, with a changing perspective over centuries: from descriptive geometric or mechanical features, to illustrations of progress in algebraic geometry, or more recently as computational/algorithmic objects of interest. We merely give below a glimpse into the subject, with a few bibliographical indications.

The simplest class of real polynomial curves is described by graphs of polynomial maps $p: \RR \longrightarrow \RR^{d-1}$, as for instance the {\it parabola}
$$ y = p(x), \ \ p(x) = a x^2 + b, \ \ a \neq 0.$$
In the above example and throughout this subsection, $(x,y)$ stand for the coordinates in affine space $\RR^2$, or if necessary to pass to complex coordinates, even in $\CC^2$.

A celebrated theorem of Abhyankar and Moh \cite{AbhyankarMoh} asserts that a smooth, rational curve $X \subset \CC^2$ admits a polynomial parametrization if and only if
it can be transformed into a straight line by invertible linear transforms and automorphisms of the form
$$ x_1 = x + h(y), \ \ y_1 =  y,\ \ h \in \CC[y].$$

One step further, Zaindenberg and Lin \cite{ZaidenbergLin} proved that any simply connected, irreducible polynomial curve in $\CC^2$ is equivalent, in the above sense, to
a {\it basic cusp} curve:
$$ x^k = y^\ell,$$ where
$k, \ell$ are relatively prime positive integers. The simple, real polynomial parametrization of this curve is $x = t^\ell, y = t^k$. As a consequence, it is interesting to note that all
2D simply connected polynomial curves have at most one singular point. In this simple situation, the ramification locus of the curve reduces to a single point
$F = \{ (0,0) \}.$

The cited landmark article by Abhyankar and Moh \cite{AbhyankarMoh} contains the following characterization: a rational curve $X \subset \CC^2$ admits a polynomial parametrization if and only if its compactification in projective space contains a single place at infinity. That means that the polynomial equation
describing the curve
$$ X = \{ (x,y) \in \CC^2; \ F(x,y) = 0 \}$$
starts with an exact power of a linear function, plus a reminder:
$$ F(x,y) = (ax+by)^d + G(x,y), \ \ |a| + |b| >0, \ \ \deg G < d.$$

A few low degree examples are in order. First, a {\it cubic with a nodal singular point} admits a polynomial parametrization:
$$ y^2 = x^2(x+1); \ \ x = t^2-1, \ y = t^3 -t.$$ Again, the ramification locus is $\{ (0,0)\}$, with the associated weight $\omega(t) = t$.
Among 2D quartics, the {\it Kampyle of Eudoxus} is a polynomial curve:
$$x^4 = a^2 (x^2 + y^2), \ \ a >0,$$
or better, in polar coordinates
$$ \rho = \frac{a}{\cos^2 \theta},$$
can be rationally parametrized as function of $t = \tan \frac{\theta}{2}$. Since the place at infinity is reduced to a point, we deduce from Abhyankar-Moh Theorem that a proper, polynomial parametrization exists.

Another quartic in two dimensions exhibits a {\it Ramphoid cusp} (that is both branches at the singular point are tangent to the same semi-axis):
$$ y^4 - 2axy^2 - 4ax^2y - ax^3 + a^2x^2 =0, \ \ a>0,$$
with parametrization
$$ x = a t^4, \ y = a(t^2 + t^3).$$

Finally, {\it l'Hospital quintic} provides another example:
$$ 64 y^5 = a (25 x^2 + 20 y^2 - 20 a y + 4a^2)^2, \ \ a>0, $$
with parametrization
$$ x = \frac{a}{2} (t - \frac{t^5}{5}), \ \ y = \frac{a}{4}(1+t^2)^2.$$
The botanics of algebraic curves is quite substantial, hiding small and big wonders. For more details we refer to the award winning website \\ {\bf https://mathcurve.com}
built and maintained by Robert Ferr\'eol.

The qualitative studies pertaining to (real) polynomial curves are also quite numerous, motivated by applications. The basic reference is the algorithmic oriented book
\cite{SWP} which treats in detail real rational curves and polynomial curves. Re-parametrizations of rational curves with one place at infinity are treated in \cite{ManochaCanny}.
A constructive approach to transforming smooth, polynomial curves into lines is taken in \cite{LamShpilrainYu}, while \cite{BenbouzianeHouariKahoui} constructs a polynomial parametrization of the same class of curves by integrating an appropriate vector field.

\end{document}